\documentclass[11pt, a4paper]{amsart}
\usepackage{amsmath, amsthm, amsfonts, amssymb}
\usepackage[english]{babel}

\input xy
\xyoption{all}

\swapnumbers
\numberwithin{equation}{section}

\theoremstyle{plain}
  \begingroup
     
     \newtheorem{theorem}[equation]{Theorem}
     
  \endgroup

\theoremstyle{definition}
   \begingroup
       \newtheorem{sinnada}[equation]{}
   \endgroup

\newcommand{\BE}{\begin{equation}}
\newcommand{\EE}{\end{equation}}

\newcommand{\cqd}{\hfill$\Box$}  

\newcommand{\cc}{\mathcal}

\newcommand{\im}{\imath}

\newcommand{\elu}{$l_u$}

\newcommand{\mr}[1]{\overset {#1} {\longrightarrow}}

\newcommand{\mmr}[1]{\overset {#1} {\hookrightarrow}}

\begin{document}

\title{On the equivalence between MV-algebras and $l$-groups with strong unit}
\author{Eduardo J. Dubuc and Yuri A. Poveda}


\maketitle

{\bf Introduction.}
This note concerns the category $\mathcal{M}$ of MV-algebras and its full subcategory $\mathcal{CM}$ of totally ordered MV-algebras, for short, \mbox{\emph{MV-chains},} and the category 
$\mathcal{G}_u$ of abelian l-groups with strong unit, for short \mbox{\emph{\elu -groups},} with  unital l-homomorphism as arrows, and its full subcategory 
$\mathcal{CG}_u$ of totally ordered abelian 
l-groups with strong unit, for short 
\mbox{\emph{\elu-chains}.}

In \cite{Ch} C.C. Chang proved that any $MV$-chain $A$ was isomorphic to the segment $A \cong \Gamma(A^*, u)$ of a 
\elu-chain $A^*$. He constructs $A^*$ by the simple intuitive idea of putting denumerable copies of $A$ on top of each other (indexed by the integers). Moreover, he also show that any \elu-chain $G$ can be recovered from its segment, since $G \cong \Gamma(G, u)^*$, 
 establishing an equivalence of categories.
In \cite{M} D. Mundici extended this result to arbitrary 
$MV$-algebras and \elu-groups. He takes a representation of $A$ as a sub-direct product of chains $A_i$, and observes that 
$A \mmr{} \prod_i A_i^*$. Then he let $A^*$ be the 
subgroup generated by $A$ inside $\prod_i A_i^*$. He proves that this idea works, and establish an equivalence of categories in a rather elaborate way by means of his concept of \emph{good sequences} and its complicated arithmetics. In this note, essentially self-contained except for Chang's result, we give a simple proof of this equivalence taking advantage directly of the arithmetics of the product 
$l$-group $\prod_i A_i^*$, avoiding entirely the notion of good sequence.    

\section{Preliminaries on $l$-groups and MV-algebras}

In this section we refer the reader to the monograph \cite{CDM} rather than to the original sources. We consider products $\;\prod_{i \in I} X_i \;$ to be the set of sections of the projection from the disjoint union 
$\;\coprod_{i \in I} X_i  \mr{\pi}  I$. 
In this way, the elements of the product are functions 
$I \mr{\sigma} \coprod\limits_{i \in I} X_i$ with values in the fibers, for $i \in I, \; \sigma(i) \in X_i$, and the algebraic and order structures are the \emph{pointwise} structures. De domain $X$ of an injective morphism 
$X \mmr{} \prod_{i \in I} X_i$ is a \emph{sub direct product} of the family $\{X_i\}_{i\in I}$. An standard application of Zorn's lemma proves that abelian 
\mbox{$l$-group} and $MV$-algebras are  subdirect products of the respective totally ordered structures. Given a  \mbox{\elu-group} 
$(G,u)$ and $MV$-algebra $A$, we have the canonical sub direct product representations:
\BE \label{canonical}
(G,u) \mmr{} 
      \textstyle\prod\limits_{P \in Sp(G)} (G/P, u_P),
\; x \mapsto \widehat{x}^g.
\hspace{3ex}
A \mmr{} \textstyle\prod\limits_{P \in Sp(A)} A/P,
\; a \mapsto \widehat{a}^v, 
\EE
where $Sp(G)$, $Sp(A)$ are the sets of all prime $l$-ideals (resp. ideals). For each $P$ we have the quotient 
$(G,u) \mr{\rho} (G/P, u_P)$,  (resp. $A \mr{\rho} A/P)$ 
 and
$\widehat{x}^g(P) = \rho(x) = [x]_P$,  (resp. $\widehat{a}^v(P) = \rho(a) = [a]_P$). Note that $\widehat{u}^g > 0$, but it is not a strong unit for the product $l$-group, which in general will not be a \elu-group (clearly when $Sp(G)$ is infinite they are unbounded sections).

\vspace{1ex}

For $G$ any abelian $l$-group, $x \in G$, set 
 $\;x^+ = 0 \vee x,\;\; x^- = 0 \vee -x\;$. Then:
\BE \label{absolutevalue}
 x = x^+ - x^-,\;\;\;  
|x| = x^+ + x^-, \;\;\;|x + y| \leq |x| + |y|.
\EE

Given any \elu-group $(G,u)$, it is not straightforward but also not difficult to check that the segment
\mbox{$\Gamma(G,u) = [0,u] = \{x \in G \;|\; 0\leq x \leq u\}$} is a MV-algebra with operations defined as follows: 
\begin{equation} \label{segment} 
x \oplus y = u \wedge (x+y), \hspace{4ex}
\neg x = u - x,
\end{equation}
and that the lattice structure induced by $G$ is the natural $MV$-algebra lattice structure.
The segment $\Gamma(G, u)$ is defined on arrows by restriction and therefore determines a functor $\cc{G}_u \mr{\Gamma} \cc{M}$. 
Recall now the following:
\begin{theorem} ${}$ \label{cignolitorrens}

a) (\cite[7.2.4]{CDM}). For any $l$-ideal $J$ of $G$, 
$J \cap [0, u]$ is an ideal of $\Gamma(G,u)$, and the morphism below is well defined and an isomorphism.  
$$
\Gamma(G,u)/(J \cap [0, u])Ê\mr{\cong} \Gamma(G/J,u_J),
\;\; [x]_{J \cap [0,u]} \longmapsto [x]_J.
$$

\vspace{1ex}

b) (\cite[7.2.2, 7.2.3]{CDM}) For any \elu-group $(G,u)$, the correspondence 
\mbox{$Sp(G) \mr{\cong} Sp(\Gamma(G,u))$,} $P \longmapsto P \cap [0,u]$, is bijective.
\end{theorem}
\begin{proof} ${}$

a) For $x \in G$, $0 \leq x \leq u$, 
$\;x \in J \iff x \in J \cap [0,u]$. This shows it is
well defined and injective. For the surjectivity, take 
$y \in G$ such that $[y]_J = [x]_J$. Then 
$[y \wedge u]_J = [y]_J \wedge [u]_J = [x]_J \wedge [u]_J 
= [x]_J$.

\vspace{1ex}

b) This correspondence holds between the sets of all ideals, it is elementary \emph{but not} straightforward, we refer the reader to \mbox{\cite[Theorem 7.2.2]{CDM}} for a detailed proof. It immediately follows from a) that $J$ is prime if and only if 
$(J \cap [0, u])$ is prime.     
\end{proof}
We establish now the key result of this paper:
\begin{theorem} \label{key}
Let $G$ be any abelian $l$-group, $u \in G, u > 0$, and let $[0,u]$ be the segment
$[0,u] = \{a \in G \:|\: 0 \leq a \leq u \}$. Let 
$[0,u]^* \subset G$ be the subgroup generated by $[0,u]$. Then:

\vspace{1ex}

a)

\vspace{-5.5ex}

$$
[0,u]^* \subseteq 
\{x \in G  \; such\;that\;  \exists\, n \geq 0,\; |x| \leq nu \}.
$$
It follows that  $[0,u]^*$ is a \elu-group with strong unit $u$.

\vspace{1ex}

b) Given any sub $MV$-algebra $A \subset [0, u]$, let 
$A^* \subset [0,u]^*$ be the subgroup generated by $A$, then $A^*$ is a \elu-group with strong unit $u$, and $A = \Gamma(A^*, u)$.

\vspace{1ex}

c)

\vspace{-5.5ex}

$$
[0,u]^* = 
\{x \in G  \; such\;that\;  \exists\, n \geq 0,\; |x| \leq nu \}.
$$
It follows that any \elu-group is generated by its segment. 
\end{theorem}
\begin{proof} 
a):
Let $x\ \in [0,u]^*$, 
$x =\sum_{k = 0}^{n-1} \,\varepsilon_k a_k$,  
  $a_k \in [0,u]$ not necessarily distinct, and 
  $\varepsilon_k = 1$ or $\varepsilon_k = -1$.  
Then $|x| \leq \sum_{k = 0}^{n-1} |\varepsilon_k a_k| = 
 \sum_{k = 0}^{n-1} a_k \leq nu$. 

\vspace{1ex}
 
b) It only remains to check that if $0 \leq x =\sum_{k = 0}^{n-1} \,a_k \leq u$, then $x \in A$. But this is clear since in this case $x = \bigoplus_{k = 0}^{n-1} \,a_k$. 

\vspace{1ex}

c)
It remains to prove the other inclusion $\;"\supseteq"\;$: Take a family $\{G_i\}_{i \in I}$ of \mbox{$l$-chains} such that $G$ is a subdirect product. We can safely assume that $G$ is a subset 
$G \;\subset  \; \prod\limits_{i \in I} G_i$. Let $x \geq 0$ and $n \geq 0$ be such that $x \leq nu$. Let for each $i \in I$ $n_i \geq 1$ be such that
$(n_i - 1)u(i) \leq x(i) \leq n_i u(i)$. Clearly 
$n_i \leq n$ for all $i$. 
Let $a_k = (x - ku) \wedge u \vee 0$, $0 \leq k \leq n-1$. Then, $(\sum_{k = 0}^{n-1} a_k)(i) =
 \sum_{k = 0}^{n-1} \,(x(i) - ku(i)) \wedge u(i) \vee 0(i)
 = (n_i - 1) u(i) + x(i) - (n_i - 1) u(i) + 0(i) = x(i)$. Thus
 $x = \sum_{k = 0}^{n-1} a_k$. But any $x$ is the difference of two positive elements, $x = x^+ - x^-$ (\ref{absolutevalue}), so the result follows.
\end{proof}

\section{Chang's equivalence}

Given any MV-chain $A$, Chang's construction consists in the set  
$$
A^* = \{(m,a) \colon m \in \mathbb{Z}, \ \text{and} \ a\in A \}
$$
ordered lexicografically, together with the following definitions: 

\BE \label{asterisco} \EE

\vspace{-12ex}
 
\begin{align*} 
     (m+1, 0) & = (m,1) \\
(m,a) + (n,b) & = (m+n, a \oplus b) &\text{if} 
                                    \ a \oplus b  <  1 \\
(m,a) + (n,b) & = (m+n+1, a \odot b ) & \text{if} \ a 
                                    \oplus b = 1 \\
       -(m,a) & = (-m-1, \neg a).
\end{align*}

In \cite[Lemma 5]{Ch} the reader can find a careful and detailed proof that $A^*$ is a \mbox{\elu-chain} with strong unit 
\mbox{$u = (0,\, 1)$}. With this, the proof of the following is immediate:
 
\begin{theorem}\cite[Lemma 6]{Ch} \label{lemma6} 
Given any MV-chain $A$ and \elu-chain $(G, \, u)$, 
the arrows:

\vspace{-6ex}

\begin{align*} 
     A \mr{\im} \Gamma(A^*,u) \subset A^*,& 
     \hspace{6ex} \im(a) =  (0,a).
   \\
    \Gamma(G, u)^* \mr{\upsilon}  (G, u) \hspace{2.5ex},&
     \hspace{1.5ex} \upsilon(m,a) = mu + a. 
\end{align*}

\vspace{-1.5ex}

\noindent are isomorphisms.
\end{theorem}
\begin{proof}
That $\im$ is an isomorphism is clear by construction of $A^*$. On the other hand, given $x \in G$, let $n_x$ be the integer such that $n_x u \leq x \leq (n_x + 1) u$. Then, the function $x \longmapsto (n_x,\, x - n_x u)$, is an inverse for $\upsilon$.
\end{proof} 
It is easy to verify that Chang isomorphisms actually determine an equivalence of categories.
\begin{theorem}\label{chainfunctor} 
Chang's construction $A^*$ can be defined on arrows and together with the segment $\Gamma$ determine functors in such a way that $\im$, 
$\upsilon$ become natural transformations. Thus they establish an equivalence of categories: 
$$
\xymatrix@1
  {
   \mathcal{CM}\; \ar[r]^-{(-)^*} & \;\mathcal{CG}_u
  }
, \hspace{4ex}      
\xymatrix@1
  {
   \mathcal{CG}_u\; \ar[r]^-\Gamma & \;\mathcal{CM}
  }
$$
\end{theorem}
\begin{proof} Given $A \mr{h} B$ in $\mathcal{CM}$,  define $A^* \mr{h^*} B^*$ by $h^*(n, a) = (n, h(a))$. Clearly the equation 
$(h_2h_1)^* = h_2^* h_1^*$ holds. From definitions \ref{asterisco} it is straightforward to check that $h^*$ is well defined and that it is a \elu-morphism. 

The naturality of $\im$ and $\upsilon$ consists on the commutativity of the following diagrams:
\BE \label{chainfunctordiagram}
\xymatrix
   {
     A \ar[d]_h \ar[r]^-\im  
  &  \Gamma(A^*, u)\:  \ar[d]^{\Gamma{(h^*)}}
                         \ar@{^{(}->}[r]
  &  A^* \ar[d]^{h^*} 
  \\ B  \ar[r]^-\im  
  &  \Gamma(B^*, u)\: \ar@{^{(}->}[r]
  & B^*
   }
\hspace{8ex}
\xymatrix
   {
     \Gamma(G, u)^* \ar[d]^{(\Gamma{\varphi})^*} 
                    \ar[r]^-\upsilon  
   & (G, u) \ar[d]^-{\varphi}  
  \\ \Gamma(H, v)^* \ar[r]^-\upsilon 
   & (H, u)
   }
\EE
that we check as follows: 
Let $a \in A$, then
$\Gamma(h^*)\, \im(a) = \Gamma(h^*)(0, a) = h^*(0, a) = (0, h(a)) = \im \, h(a).
$
Let $(n, x) \in \Gamma(G, u)^*$, then
$
\varphi \, \upsilon(n,x) =
\varphi(nu + x) = nu + \varphi(x) = 
\upsilon(n,\varphi(x)) =
\upsilon(n, (\Gamma\varphi)(x)) = 
\upsilon \,(\Gamma\varphi)^*(n, x).  
$
\end{proof}

\section{Extension of the functors, $(-)^*$ and $\Gamma$ }

Consider the following diagrams:
$$
\xymatrix@C10ex
  {
   *\txt{$\mathcal{CM}\;$ \\ { }} \ar@<1ex>[r]^-{(-)^*}      
                                  \ar@{^{(}->}[d]
   & 
   *\txt{$\;\,\mathcal{CG}_u$ \\ { }} \ar@{^{(}->}[d]
  \\
   \mathcal{M}\;  
   & \;\mathcal{G}_u 
  }
, \hspace{6ex}     
\xymatrix@C10ex
  {
   *\txt{$\mathcal{CG}_u \;$ \\ { }} \ar@<1ex>[r]^-\Gamma 
                                     \ar@{^{(}->}[d]
   & 
   *\txt{$\;\,\mathcal{CM}\;$ \\ { }} \ar@{^{(}->}[d] 
  \\
   \mathcal{G}_u\; \ar[r]^-\Gamma & \;\mathcal{M} 
  }
$$
We will complete the diagram on the left by extending the Chang construction into a functor 
$\mathcal{M}\; \mr{(-)^*} \; \mathcal{G}_u$, and prove that this extends the equivalence in the first row into an equivalence in the second row.

\vspace{1ex}

Given any $MV$-algebra $A$ we will define the \elu-group $A^*$ as a subdirect product of the family of 
\elu-chains $\{(A/P)^*\}_{P \in Sp(A)}$.
We refer to the canonical subdirect product representation of $A$ (\ref{canonical}) and to Theorem \ref{lemma6}. Consider the diagram:
$$
\xymatrix@R=1ex
      {
       *\txt{$A$ \;\;\; \\ {\;}} 
                     \ar@<1ex>@{>->}[r]^-{\widehat{(-)}^v} 
                     \ar@{>->}@(d,l)[dr]^\cong
     & *\txt{$\textstyle\prod\limits_{P \in Sp(A)} 
                                 (A/P)$ \;\; \\ \;}
                            \ar@<1ex>@{>->}[r]^{\im}
     & *\txt{$\textstyle\prod\limits_{P \in Sp(A)} 
                               (A/P)^*$ \;\; \\ \;}
     \\
     & {\;{A^\circ} \subset A^* \;}
       \ar@{^{(}->}@(r,d)[ur]
      }
$$
where ${A^\circ}$ denotes the image subset of the composite arrow $\im \, \widehat{(-)}^v$, and we let $A^*$ be the subgroup generated  by ${A^\circ}$, and 
$u = \im \, \widehat{1}^v$. For $a \in A$, 
\mbox{$\im\, \widehat{a}^v(P) = \im([a]_P)$,} thus 
$\im\,\widehat{a}^v \leq u$, i.e. ${A^\circ} \subset [0, u]$. From Theorem \ref{key} b) we have:
\begin{theorem}
$A^*$ is a \elu-group with strong unit $u$, and 
${A^\circ} = \Gamma(A^*, u)$. It follows that 
$A \mr{\im \,\widehat{(-)}^v} \Gamma(A^*, u)$ is an isomorphism.
\cqd
\end{theorem}
\begin{theorem} \label{upsiloniso}
Chang isomorphisms $\upsilon$ in Theorem \ref{lemma6} on the fibers, restricted to $\Gamma(G, u)^*$, determine an isomorphism 
$\Gamma(G, u)^* \mr{\upsilon} (G, u)$. 
\end{theorem} 
\begin{proof}
Consider the diagram:
$$
\xymatrix
   {
    \Gamma(G, u)^*\; \ar@{-->}[d]^{\upsilon} 
               \ar@<-0.5ex>@{^{(}->}[r]
  & \prod\limits_{P \in Sp(G)} \,(\Gamma(G, u)/P
      \cap[0,u])^* \ar[r]^-{\cong}
  & \prod\limits_{P \in Sp(G)} \,\Gamma(G/P, [u]_P)^* 
      \ar[d]^-{\upsilon}_-{\cong}
  \\
    (G, u)\;\; \ar@{>->}[rr]^-{\widehat{(-)}^g}
  && \prod\limits_{P \in Sp(G)} \,(G/P, [u]_P)
  }
$$
The arrows on the top are justified by theorem \ref{cignolitorrens} b) and a) respectively.  We show now that going all the way right and down in the diagram factors through $(G, u)$. It is enough to consider generators 
\mbox{$\im\,\widehat{x}^v \in {\Gamma(G, u)^\circ}$,} 
\mbox{$x \in \Gamma(G,u)$.} Then: $\im\,\widehat{x}^v(P) =
\im([x]_{P \cap [0, u]})$, wich by 
\ref{cignolitorrens} a) corresponds to  
$\im([x]_{P})$ which in turn goes down to 
$\upsilon \,\im([x]_{P}) = [x]_P = \widehat{x}^g(P)$. It is clear that $\Gamma(G, u)^* \mr{\upsilon} (G, u)$ so determined is injective, and the surjectivity follows immediately from theorem \ref{key}, c).

\vspace{-3ex}

\BE \label{igual}
\hspace{-18ex} \txt{Remark that for any $x \in \Gamma(G,u)$, 
$\upsilon \, \im \, \widehat{x}^v = x$.}
\EE

\vspace{-3ex}

\end{proof}
We show now that the construction of $A^*$ can be defined on arrows in a functorial way. This follows by standard  techniques in representation theory.

\begin{theorem}
Let $A \mr{h} B$ a morphism of $MV$-algebras. Then there is a morphism of \elu-groups $A^* \mr{h^*} B^*$ such that the following diagram commutes:
\BE \label{h^*}
\xymatrix
   {
    A \ar[d]_h \ar[r]^-\cong 
  & {A^\circ}\: \ar@{-->}[d]^{h^*}
                  \ar@{^{(}->}[r]^-{=} 
  & \Gamma(A^*, u)\:  \ar@{-->}[d]^{\Gamma{(h^*)}}
                        \ar@{^{(}->}[r]
  & A^*\: \ar@{-->}[d]^{h^*}  \ar@{^{(}->}[r]
  & \textstyle\prod\limits_{P \in Sp(A)} (A/P)^*
                        \ar[d]^{h^*}
  \\ 
    B  \ar[r]^-\cong 
  & {B^\circ}\: \ar@{^{(}->}[r]^-{=}
  & \Gamma(B^*, u)\: \ar@{^{(}->}[r]
  & B^*\:  \ar@{^{(}->}[r]
  & \textstyle\prod\limits_{P \in Sp(B)} (B/P)^*
   }
\EE
This defines a functor $\cc{M} \mr{(-)^*} \cc{G}_u$ in such a way that the isomorphisms 
 \mbox{$A \mr{\im \,\widehat{(-)}^v} \Gamma(A^*, u)$} and 
 $\Gamma(G, u)^* \mr{\upsilon} (G, u)$ become natural transformations.
\end{theorem}
\begin{proof}
Let $P \in Sp(B)$. Then $h^{-1}P \in Sp{A}$ and there is a $MV$-algebra morphism $A/h^{-1}P \mr{h|_P} B/P$ well-defined by 
$h|_{P}([a]_{h^{-1}P}) = [h(a)]_P$. From the left diagram in
(\ref{chainfunctordiagram}) we have a commutative diagram:
$$
\xymatrix
   {
    A/h^{-1}P      \ar[d]_{h|_P} \ar[r]^-\im  
  & (A/h^{-1}P)^*  \ar[d]^{(h|_P)^*}
  \\    
    B/P  \ar[r]^-\im  
  & (B/P)^*
   }
$$
We define   
the right most vertical arrow in diagram (\ref{h^*}) by \mbox{$h^*(\sigma)(P) = (h|_{P})^*(\sigma(h^{-1}P)$}, for any 
$\sigma \in \textstyle\prod\limits_{P \in Sp(A)} (A/P)^*$. It is straightforward to check that $h^*$ is a morphism of $l$-groups and that the equation $(h_2h_1)^* = h_2^* h_1^*$ holds.
We prove now that it restricts to $A^* \mr{} B^*$. It is enough to check that it restricts to 
${A^\circ} \mr{} {B^\circ}$, which amounts to show that the exterior of diagram \ref{h^*} commutes: 
Let $a \in A$ and $P \in S_p(B)$, then
$h^*(\im \, \widehat{a}^v)(P) = 
(h|_{P})^*(\im \, \widehat{a}^v(h^{-1}P) =
 (h|_{P})^*(\im[a]_{h^{-1}P}) = \im h|_P ([a]_{h^{-1}P}) = 
 \im [h(a)]_P = \im \widehat{h(a)}^v(P)$, thus  $h^*(\im \,\widehat{a}^v) = \im \, \widehat{h(a)}^v$.
 It only remains to show that $\upsilon$ is a natural transformation, that is, for any $(G, u) \mr{\varphi} (H, u)$, 
$\upsilon\,(\Gamma\varphi)^* = \varphi \, \upsilon$, see the right  diagram in (\ref{chainfunctordiagram}). It is enough to check the equation on the generators   
\mbox{$\im\,\widehat{x}^v \in {\Gamma(G, u)^\circ}$,} 
\mbox{$x \in \Gamma(G,u)$.} Then:
$\upsilon\,(\Gamma\varphi)^* \im\,\widehat{x}^v = 
\upsilon \, \im \, \widehat{(\Gamma \varphi)x}^v =
\upsilon \, \im \, \widehat{\varphi x}^v =  \varphi x =
\varphi \, \upsilon \, \im \, \widehat{x}^v$, this last two equalities by (\ref{igual}) in Theorem \ref{upsiloniso}.
 \end{proof}

 
 \begin{sinnada} {\bf A final comment.} This finishes the proof that the functor 
 $\cc{G}_u \mr{\Gamma} \cc{M}$ establishes an equivalence of categories, by constructing explicitly a quasi-inverse 
 $\cc{M} \mr{(-)^*} \cc{G}_u$ directly related to Chang's original construction. Namely, we use the \mbox{$l$-group} determined pointwise by the Chang $l$-groups constructed on the fibers of the canonical subdirect product representation, and profit directly to be inside this large $l$-group in order to prove that the construction actually determines an equivalence of categories.
\end{sinnada} 
\begin{sinnada}
{\bf A word on the role of good-sequences.}
We refer to Theorem \ref{key}. Looking at the proof of item c), it is immediate to check by evaluating on each $i \in I$, that the sequence 
$a_k \in [0,u]$ that represents $x$ as $x = \sum_{k = 0}^{n-1} a_k$, $a_k = (x - ku) \wedge u \vee 0$, $0 \leq k \leq n-1$, satisfies $a_k \oplus a_{k+1} = a_k$. That is, it is a (canonical) good sequence \mbox{attached} to $x$, and it the unique such that represents $x$. This determines a bijection between the set of good sequences of $A = [0,u]$ and the \elu-group 
\mbox{$ [0, u]^* =
\{x \in G  \; such\;that\;  \exists\, n \geq 0,\; |x| \leq nu \}
\; \subset \; G$.} This bijection transports the structure of $G$ into a \mbox{\elu-group} structure on the set of good sequences, which is cumbersome and alien to previous intuition. In our proof of the theorem we work directly in the 
\mbox{\elu-group} 
$A^* \: \subset \: G = \prod_P A_P^*$, $P \in Sp(A)$, $A_P = A/P$, 
 instead of working with the good sequences and its complicated arithmetic. The subdirect representation theorem is anyway essential in both approaches as we know them today. 
 
It would be interesting to find a constructive proof (no Zorn's lemma) of this equivalence. Note that the construction 
$A^* = Gs(A)$ (set  of  good  sequences) is there, and the definition of its arithmetics and strong unit is constructive. Probably we have an isomorphism of abelian groups 
\mbox{$Gs(A) \cong Free(|A|)/\sim$,} the quotient of the free abelian group on the underlying set of $A$ by the congruence generated by the pairs 
$(a,b) \sim (a \oplus b, a \odot b)$. In this way we actually  have $A^*$, the challenge is to prove that it is an $l$-group, and that it is the required \elu-group that yields the equivalence, without using that there are enough primes. 
\end{sinnada}

\end{document}